\documentclass[a4paper, 12pt]{article}
\usepackage[cp1251]{inputenc}
\usepackage[english]{babel}
\usepackage{amsthm}
\usepackage{cite}
\usepackage{tikz-cd}
\usepackage{amsmath}
\usepackage{amsfonts}
\usepackage{amssymb}
\usepackage{hyperref}
\usepackage[]{authblk}
\usepackage{epsfig,graphicx}

\DeclareMathOperator{\Aut}{Aut}

\DeclareMathOperator{\Ker}{Ker}

\DeclareMathOperator{\Deg}{deg}

\DeclareMathOperator{\diag}{diag}

\DeclareMathOperator{\Ind}{Ind}

\newtheorem{thm}{Theorem}[section]
\newtheorem{lem}[thm]{Lemma}

\newtheorem{prop}[thm]{Proposition}
\newtheorem{cor}[thm]{Corollary}
\newtheorem{conj}[thm]{Conjecture}
\newtheorem{Def}[thm]{Definition}
\newtheorem{remark}[thm]{Remark}
\begin{document}
\renewcommand{\thefootnote}{\fnsymbol{footnote}}
\renewcommand{\thefootnote}{\arabic{footnote}}
\fontsize{12}{12pt}\selectfont
\title{\bf Noncommutative Bia\l{}ynicki-Birula Theorem}
\renewcommand\Affilfont{\itshape\small}
\author[1,4]{Andrey Elishev\thanks{elishev@phystech.edu}}
\author[5]{Alexei Kanel-Belov\thanks{kanel@mccme.ru}}
\author[1,3]{Farrokh Razavinia\thanks{f.razavinia@phystech.edu}}
\author[5]{Jie-Tai Yu\thanks{jietaiyu@szu.edu.cn}}
\author[2]{Wenchao Zhang\thanks{whzecomjm@gmail.com}}

\affil[1]{Laboratory of Advanced Combinatorics and Network Applications, Moscow Institute of Physics and Technology, Dolgoprudny, Moscow Region, 141700, Russia}
\affil[2]{Mathematics Department, Bar-Ilan University, Ramat-Gan, 52900, Israel}
\affil[3]{Department of Mathematics, University of Porto, Pra\c{c}a de Gomes Teixeira, 4099-002 Porto, Portugal}
\affil[4]{Department of Discrete Mathematics, Moscow Institute of Physics and Technology, Dolgoprudny, Moscow Region, 141700, Russia}
\affil[5]{College of Mathematics and Statistics, Shenzhen University, Shenzhen, 518061, China}

\date{}

\maketitle

\renewcommand{\abstractname}{Abstract}
\begin{abstract}
In this short note we prove that every maximal torus action on the free algebra is conjugate to a linear action. This statement is the free algebra analogue of a classical theorem of A. Bia\l{}ynicki-Birula. The paper was supported by russian science foundation grant N 17-11-01377

\end{abstract}

\section{Actions of algebraic tori}

In this note we consider algebraic torus actions on the affine space, according to Bia\l{}ynicki-Birula, and formulate certain noncommutative generalizations.

We begin by recalling a few basic definitions.
Let $\mathbb{K}$ be an algebraically closed field.

\begin{Def}
	An algebraic group is a variety $G$ equipped with the structure of a group, such that the multiplication map $m: G \times G \to G:(g_1,g_2)\mapsto g_1g_2$ and the inverse map $\iota:G \to G:g \mapsto g^{-1}$ are morphisms of varieties.
\end{Def}
\begin{Def}
	A $G$-variety is a variety equipped with an action of the algebraic group $G$,
	$$\alpha:G \times X \to X:(g,x)\mapsto g \cdot x,$$
	which is also a morphism of varieties. We then say that $\alpha$ is an algebraic $G$-action.
\end{Def}

Let $\mathbb{K}$ be our ground field, which is assumed to be algebraically closed. Let $Z=\{z_1, z_2, \ldots \}=\{z_i: i\in I\}$ be a finite or a countable set of variables (where $I=\{1,2, \ldots\}$ is an index set),  and let $Z^*$ denote the free semigroup generated by $Z$, $Z^{+}=Z^*\backslash \{1\}$. Moreover let $F_I(\mathbb{K})=\mathbb{K}\left\langle  Z \right\rangle $ be the free associative $\mathbb{K}$-algebra and $\hat{F}_I(\mathbb{K})=\mathbb{K}\left\langle \left\langle Z \right\rangle \right\rangle $ be the algebra of formal power series in free variables.

Denote by $\mathcal{W}=\left\langle Z \right\rangle $ the free monoid of words over the alphabet $Z$ (with 1 as the empty word) such that $\left| \mathcal{W} \right| \ge 1 $, for $\left| \mathcal{W} \right|$ the length of the word $\mathcal{W} \in Z^+$.

For an alphabet $Z$, the free associative $\mathbb{K}$-algebra on $Z$ is
$$\mathbb{K}\left\langle Z \right\rangle :=\oplus_{\mathcal{W}\in Z^*} \mathbb{K}\mathcal{W},$$ where the multiplication is $\mathbb{K}$-bilinear extension of the concatenation on words,  $Z^*$  denotes the free monoid on $Z$, and $\mathbb{K}\mathcal{W}$ denotes the free $\mathbb{K}$-module on one element, the word $\mathcal{W}$.  Any element of $\mathbb{K}\left\langle Z \right\rangle $ can thus be written uniquely in the form
$$\sum\limits_{k=0}^{\infty} \sum\limits_{i_1,\ldots,i_k \in I}^{}a_{i_1,i_2,\ldots,i_k}z_{i_1}z_{i_2}\ldots z_{i_k},$$
where the coefficients $a_{i_1,i_2,\ldots,i_k}$ are elements of the field $\mathbb{K}$ and all but finitely many of these elements are zero.

In our context, the alphabet $Z$ is the same as the set of algebra generators, therefore the terms "monomial" and "word" will be used interchangeably.

In the sequel, we employ a (slightly ambiguous) short-hand notation for a free algebra monomial. For an element $z$, its powers are defined as usual. Any monomial $z_{i_1}z_{i_2}\ldots z_{i_k}$ can then be written in a reduced form with subwords $zz\ldots z$ replaced by powers.

We then write
$$
z^I = z_{j_1}^{i_1}z_{j_2}^{i_2}\ldots z_{j_k}^{i_k}
$$
where by $I$ we mean an assignment of $i_k$ to $j_k$ in the word $z^I$. Sometimes we refer to $I$ as a multi-index, although the term is not entirely accurate. If $I$ is such a multi-index, its abosulte value $|I|$ is defined as the sum $i_1+\cdots+ i_k$.

\smallskip

For a field $\mathbb{K}$, let $\mathbb{K}^{\times}=\mathbb{K}\backslash \{0\}$ denote the multiplicative group of its non-zero elements viewed as an algebraic $\mathbb{K}$-group.

\begin{Def} \label{defgroup}
An $n$-dimensional algebraic $\mathbb{K}$-torus is a group
$$
\mathbb{T}_n\simeq (\mathbb{K}^{\times})^n
$$
(with obvious multiplication).
\end{Def}
Denote by $\mathbb{A}^n$ the affine space of dimension $n$ over $\mathbb{K}$.
\begin{Def} \label{defaction}
A (left) torus action is a morphism
$$
\sigma: \mathbb{T}_n\times \mathbb{A}^n\rightarrow \mathbb{A}^n.
$$
that fulfills the usual axioms (identity and compatibility):
$$
\sigma(1,x)=x,\;\;\sigma(t_1,\sigma(t_2,x))=\sigma(t_1t_2,x).
$$

The action $\sigma$ is \textbf{effective} if for every $t\neq 1$ there is an element $x\in \mathbb{A}^n$ such that $\sigma(t,x)\neq x$.
\end{Def}

In \cite{BB1}, Bia\l{}ynicki-Birula proved the following two theorems.

\begin{thm} \label{BBthm1}
Any regular action of $\mathbb{T}_n$ on $\mathbb{A}^n$ has a fixed point.
\end{thm}

\begin{thm} \label{BBthm2}
Any effective and regular action of $\mathbb{T}_n$ on $\mathbb{A}^n$ is a representation in some coordinate system.
\end{thm}

The term "regular" is to be understood here as in the algebro-geometric context of regular function (Bia\l{}ynicki-Birula also considered birational actions).
The last theorem says that any effective regular maximal torus action on the affine space is conjugate to a linear action, or, as it is sometimes called, \textbf{linearizable}.

\medskip

An algebraic group action on $\mathbb{A}^n$ is the same as an action by automorphisms on the algebra
$$
\mathbb{K}[x_1,\ldots,x_n]
$$
of global sections of the structure sheaf. In other words, it is a homomorphism
$$
\sigma: \mathbb{T}_n\rightarrow \Aut \mathbb{K}[x_1,\ldots,x_n].
$$
An action is effective iff $\Ker\sigma = \lbrace 1\rbrace$.

The polynomial algebra is a quotient of the free associative algebra
$$
F_n = \mathbb{K}\langle z_1,\ldots,z_n\rangle
$$
by the commutator ideal $I$ (it is the two-sided ideal generated by all elements of the form $fg-gf$). From the standpoint of Noncommutative geometry, the algebra $\Gamma(X,\mathcal{O}_X)$ of global sections (along with the category of f.g. projective modules) contains all the relevant topological data of $X$, and various non-commutative algebras (PI-algebras) may be thought of as global function algebras over "noncommutative spaces". Therefore, noncommutative analogue of the Bia\l{}ynicki-Birula theorem is a subject of legitimate interest.

\smallskip

In this short note we establish the free algebra version of the Bia\l{}ynicki-Birula theorem. The latter is formulated as follows.
\begin{thm} \label{BBfree}
Suppose given an action $\sigma$ of the algebraic $n$-torus $\mathbb{T}_n$ on the free algebra $F_n$. If $\sigma$ is effective, then it is linearizable.
\end{thm}

\smallskip

The linearization problem, as it has become known since Kambayashi, asks whether all (effective, regular) actions of a given type of algebraic groups on the affine space of given dimension are conjugate to representations. According to Theorem \ref{BBfree}, the linearization problem extends to the noncommutative category. Several known results concerning the (commutative) linearization problem are summarized below.

\begin{enumerate}
	\item Any effective regular torus action on $\mathbb{A}^2$ is linearizable (Gutwirth \cite{RefB}).
	\item Any effective regular torus action on $\mathbb{A}^n$ has a fixed point (Bia\l{}ynicki-Birula \cite{BB1}).
	\item Any effective regular action of $\mathbb{T}_{n-1}$ on $\mathbb{A}^n$ is linearizable (Bia\l{}ynicki-Birula \cite{BB2}).
    \item Any (effective, regular) one-dimensional torus action (i.e., action of $\mathbb{K}^{\times})$ on $\mathbb{A}^3$ is linearizable (Koras and Russell \cite{KoRu2}).
	\item If the ground field is not algebraically closed, then a torus action on $\mathbb{A}^n$ need not be linearizable. In \cite{Asanuma}, Asanuma proved that over any field $\mathbb{K}$, if there exists a non-rectifiable closed embedding from $\mathbb{A}^{m}$ into $\mathbb A^{n}$, then there exist non-linearizable effective actions of $(\mathbb{K}^{\times})^r$ on $\mathbb A^{1+n+m}$ for $1\le r\le 1+m$.
	\item When $\mathbb {K}$ is infinite and has positive characteristic, there are examples of non-linearizable torus actions on $\mathbb{A}^{n}$ (Asanuma \cite{Asanuma}).
\end{enumerate}

\begin{remark} \label{nonessential1}
A closed embedding $\iota:\mathbb A^m\to\mathbb A^n$ is said to be rectifiable if it is conjugate to a linear embedding by an automorphism of $\mathbb A^n$.
\end{remark}

As can be inferred from the review above, the context of the linearization problem is rather broad, even in the case of torus actions. The regulating parameters are the dimensions of the torus and the affine space. This situation is due to the fact that the general form of the linearization conjecture (i.e., the conjecture that states that any effective regular torus action on any affine space is linearizable) has a negative answer.

\smallskip

Transition to the noncommutative geometry presents the inquirer with an even broader context: one now may vary the dimensions as well as impose restrictions on the action in the form of preservation of the PI-identities. Caution is well advised. Some of the results are generalized in a straightforward manner -- the main theorem of this paper being the typical example, others require more subtlety and effort (cf. \ref{conjlin} and the discussion at the end of the note). Of some note to us, given our ongoing work in deformation quantization (see, for instance, \cite{KGE}) is the following instance of the linearization problem, which we formulate as a conjecture.

\begin{conj} \label{BBsympl}
For $n\geq 1$, let $P_n$ denote the commutative Poisson algebra, i.e. the polynomial algebra
$$
\mathbb{K}[z_1,\ldots,z_{2n}]
$$
equipped with the Poisson bracket defined by
$$
\lbrace z_i, z_j\rbrace = \delta_{i,n+j}-\delta_{i+n,j}.
$$
Then any effective regular action of $\mathbb{T}_n$ by automorphisms of $P_n$ is linearizable.
\end{conj}

It is interesting to note that the context of Conjecture \ref{BBsympl} admits a vague analogy in the real transcendental category (with $P_n$ replaced by an appropriate algebra of smooth functions, cf. for instance the work of Zung \cite{Zung}). Although the instances of the linearization problem we consider in this note, as well as the original theorem of Bia\l{}ynicki-Birula, are essentially of complex algebraic nature, it may be worthwhile to search for analytic analogues of the real transcendental linearization (however whether this will give a feasible approach to Conjecture \ref{BBsympl} is unclear, the hurdles being significant and fairly obvious).

\subsection*{Acknowledgments}

The main result of this note was conceived in the prior work \cite{KBYu} of A. K.-B., J.-T. Y. and A. E.. Theorem \ref{BBfree} is due to A. E. and A. K.-B.; Lemma \ref{fixedorigin} and the review of known results for the linearization problem is due to F. R., J.-T. Y. and W. Z..

F. R. is also is responsible for the investigation of possible transcendental analogies.

A. E. and A. K.-B. are supported by the Russian Science Foundation grant No. 17-11-01377.

F. R. is supported by the FCT (Foundation for Science and Technology of Portugal) scholarship with reference number PD/BD/142959/2018.

\section{Proof of Theorem \ref{BBfree}}

The proof proceeds along the lines of the original commutative case proof of Bia\l{}ynicki-Birula.

If $\sigma$ is the effective action of Theorem \ref{BBfree}, then for each $t\in \mathbb{T}_n$ the automorphism
$$
\sigma(t): F_n\rightarrow F_n
$$
is given by the $n$-tuple of images of the generators $z_1,\ldots,z_n$ of the free algebra:
$$
(f_1(t,z_1,\ldots,z_n),\ldots,f_n(t,z_1,\ldots,z_n)).
$$
Each of the $f_1,\ldots, f_n$ is a polynomial in the free variables.

\begin{lem} \label{fixedorigin}
There is a translation of the free generators
$$
(z_1,\ldots,z_n)\rightarrow (z_1-c_1,\ldots,z_n-c_n),\;\;(c_i\in\mathbb{K})
$$
such that (for all $t\in\mathbb{T}_n$) the polynomials $f_i(t,z_1-c_1,\ldots,z_n-c_n)$ have zero free part.
\end{lem}
\begin{proof}
This is a direct corollary of Theorem \ref{BBthm1}. Indeed, any action $\sigma$ on the free algebra induces, by taking the canonical projection with respect to the commutator ideal $I$, an action $\bar{\sigma}$ on the commutative algebra $\mathbb{K}[x_1,\ldots,x_n]$. If $\sigma$ is regular, then so is $\bar{\sigma}$. By Theorem \ref{BBthm1}, $\bar{\sigma}$ (or rather, its geometric counterpart) has a fixed point, therefore the images of commutative generators $x_i$ under $\bar{\sigma}(t)$ (for every $t$) will be polynomials with trivial degree-zero part. Consequently, the same will hold for $\sigma$.
\end{proof}

We may then suppose, without loss of generality, that the polynomials $f_i$ have the form
$$
f_i(t,z_1,\ldots,z_n)=\sum_{j=1}^{n}a_{ij}(t)z_j + \sum_{j,l=1}^{n}a_{ijl}(t)z_jz_l + \sum_{k=3}^{N}\sum_{J,|J|=k}a_{i,J}(t)z^J
$$
where by $z^J$ we denote, as in the introduction, a particular monomial
$$
z_{i_1}^{k_1}z_{i_2}^{k_2}\ldots
$$
(a word in the alphabet $\lbrace z_1,\ldots, z_n\rbrace$ in the reduced notation; $J$ is the multi-index in the sense described above);
also, $N$ is the degree of the automorphism (which is finite) and $a_{ij}, a_{ijl},\ldots$ are polynomials in $t_1,\ldots, t_n$.

As $\sigma_t$ is an automorphism, the matrix $[a_{ij}]$ that determines the linear part is non-singular. Therefore, without loss of generality we may assume it to be diagonal (just as in the commutative case \cite{BB1}) of the form
$$
\diag(t_1^{m_{11}}\ldots t_n^{m_{1n}},\ldots, t_1^{m_{n1}}\ldots t_n^{m_{nn}}).
$$

Now, just as in \cite{BB1}, we have the following
\begin{lem} \label{lem1}
The power matrix $[m_{ij}]$ is non-singular.
\end{lem}
\begin{proof}
Consider a linear action $\tau$ defined by
$$
\tau(t):(z_1,\ldots, z_n)\mapsto (t_1^{m_{11}}\ldots t_n^{m_{1n}}z_1,\ldots, t_1^{m_{n1}}\ldots t_n^{m_{nn}}z_n),\;\; (t_1,\ldots,t_n)\in\mathbb{T}_n.
$$
If $T_1\subset T_n$ is any one-dimensional torus, the restriction of $\tau$ to $\mathbb{T}_1$ is non-trivial. Indeed, were it to happen that for some $\mathbb{T}_1$,
$$
\tau(t)z=z,\;\;t\in \mathbb{T}_1,\;\;(z=(z_1,\ldots,z_n))
$$
then our initial action $\sigma$, whose linear part is represented by $\tau$, would be identity modulo terms of degree $>1$:
$$
\sigma(t)(z_i) = z_i + \sum_{j,l}a_{ijl}(t)z_jz_l+\cdots.
$$
Now, equality $\sigma(t^2)(z)=\sigma(t)(\sigma(t)(z))$ implies
\begin{align*}
\sigma(t)(\sigma(t)(z_i))&=\sigma(t)\left(z_i+\sum_{jl}a_{ijl}(t)z_jz_l+\cdots\right) \\&=
z_i+ \sum_{jl}a_{ijl}(t)z_jz_l+\sum_{jl}a_{ijl}(t)(z_j+\sum_{km}a_{jkm}(t)z_kz_m+\cdots)\\&(z_l+\sum_{k'm'}a_{lk'm'}(t)z_{k'}z_{m'}+\cdots)+\cdots\\&=
z_i+\sum_{jl}a_{ijl}(t^2)z_jz_l+\cdots
\end{align*}
which means that
$$
2a_{ijl}(t)=a_{ijl}(t^2)
$$
and therefore $a_{ijl}(t)=0$. The coefficients of the higher-degree terms are processed by induction (on the total degree of the monomial). Thus
$$
\sigma(t)(z) = z,\;\;t\in\mathbb{T}_1
$$
which is a contradiction since $\sigma$ is effective. Finally, if $[m_{ij}]$ were singular, then one would easily find a one-dimensional torus such that the restriction of $\tau$ were trivial.
\end{proof}

Consider the action
$$
\varphi(t) = \tau(t^{-1})\circ\sigma(t).
$$
The images under $\varphi(t)$ are
$$
(g_1(z,t),\ldots, g_n(z,t)),\;\;(t = (t_1,\ldots,t_n))
$$
with
$$
g_i(z,t) = \sum g_{i,m_1\ldots m_n}(z)t_1^{m_1}\ldots t_n^{m_n},\;\;m_1,\ldots, m_n\in\mathbb{Z}.
$$
Define $G_i(z) = g_{i,0\ldots 0}(z)$ and consider the map $\beta:F_n\rightarrow F_n$,
$$
\beta:(z_1,\ldots,z_n)\mapsto (G_1(z),\ldots, G_n(z)).
$$
\begin{lem} \label{lem2}
$\beta\in \Aut F_n$ and
$$
\beta = \tau(t^{-1})\circ\beta\circ\sigma(t).
$$
\end{lem}
\begin{proof}
This lemma mirrors the final part in the proof in \cite{BB1}. The conjugation is straightforward, since for every $s,t\in\mathbb{T}_n$ one has
$$
\varphi(st) = \tau(t^{-1}s^{-1})\circ\sigma(st) = \tau(t^{-1})\circ\tau(s^{-1})\circ\sigma(s)\circ\sigma(t)=\tau(t^{-1})\circ\varphi(s)\circ\sigma(t).
$$

Denote by $\hat{F}_n$ the power series completion of the free algebra $F_n$, and let $\hat{\sigma}$, $\hat{\tau}$ and $\hat{\beta}$ denote the endomorphisms of the power series algebra induced by corresponding morphisms of $F_n$. The endomorphisms $\hat{\sigma}$, $\hat{\tau}$, $\hat{\beta}$ come from (polynomial) automorphisms and therefore are invertible.

Let
$$
\hat{\beta}^{-1}(z_i) \equiv B_i(z) = \sum_{J}b_{i,J}z^J
$$
(just as before, $z^J$ is the monomial with multi-index $J$). Then
$$
\hat{\beta}\circ\hat{\tau}(t)\circ\hat{\beta}^{-1}(z_i) = B_i(t_1^{m_{11}}\ldots t_n^{m_{1n}}G_1(z),\ldots,t_1^{m_{n1}}\ldots t_n^{m_{nn}}G_n(z)).
$$
Now, from the conjugation property we must have
$$
\hat{\beta}=\hat{\sigma}(t^{-1})\circ\hat{\beta}\circ\hat{\tau}(t),
$$
therefore $\hat{\sigma}(t) = \hat{\beta}\circ\hat{\tau}(t)\circ\hat{\beta}^{-1}$ and
$$
\hat{\sigma}(t)(z_i) = \sum_{J}b_{i,J}(t_1^{m_{11}}\ldots t_n^{m_{1n}})^{j_1}\ldots (t_1^{m_{n1}}\ldots t_n^{m_{nn}})^{j_n}G(z)^J;
$$
here the notation $G(z)^J$ stands for a word in $G_i(z)$ with multi-index $J$, while the exponents $j_1,\ldots, j_n$ count how many times a given index appears in $J$ (or, equivalently, how many times a given generator $z_i$ appears in the word $z^J$).

Therefore, the coefficient of $\hat{\sigma}(t)(z_i)$ at $z^J$ has the form
$$
b_{i,J}(t_1^{m_{11}}\ldots t_n^{m_{1n}})^{j_1}\ldots (t_1^{m_{n1}}\ldots t_n^{m_{nn}})^{j_n}+S
$$
with $S$ a finite sum of monomials of the form
$$
c_L (t_1^{m_{11}}\ldots t_n^{m_{1n}})^{l_1}\ldots (t_1^{m_{n1}}\ldots t_n^{m_{nn}})^{l_n}
$$
with $(j_1,\ldots,j_n)\neq (l_1,\ldots,l_n)$. Since the power matrix $[m_{ij}]$ is non-singular, if $b_i,J\neq 0$, we can find a $t\in\mathbb{T}_n$ such that the coefficient is not zero. Since $\sigma$ is an algebraic action, the degree
$$
\sup_{t}\Deg (\hat{\sigma})
$$
is a finite integer $N$. With the previous statement, this implies that
$$
b_{i,J} = 0,\;\;\text{whenever}\;\;|J|>N.
$$
Therefore, $B_i(z)$ are polynomials in the free variables. What remains is to notice that
$$
z_i = B_i(G_1(z),\ldots,G_n(z)).
$$
Thus $\beta$ is an automorphism.
\end{proof}
From Lemma \ref{lem2} it follows that
$$
\tau(t) = \beta^{-1}\circ \sigma(t)\circ\beta
$$
which is the linearization of $\sigma$. Theorem \ref{BBfree} is proved.

\section{Discussion}

The noncommutative toric action linearization theorem that we have proved has several useful applications. In the work \cite{KBYu}, it is used to investigate the properties of the group $\Aut F_n$ of automorphisms of the free algebra. As a corollary of Theorem \ref{BBfree}, one gets
\begin{cor} \label{cor1}
Let $\theta$ denote the standard action of $\mathbb{T}_n$ on
$K[x_1,\ldots,x_n]$ -- i.e., the action
$$
\theta_t: (x_1,\ldots,x_n)\mapsto (t_1x_1,\ldots,t_nx_n).
$$
 Let $\tilde{\theta}$ denote its lifting to an action on the free associative algebra $F_n$. Then
$\tilde{\theta}$ is also given by the standard torus action.
\end{cor}

This statement plays a part, along with a number of results concerning the induced formal power series topology on $\Aut F_n$, in the establishment of the following proposition (cf. \cite{KBYu}).
 \begin{prop}        \label{ThAutAutFree}
When $n\geq 3$, any $\Ind$-scheme automorphism $\varphi$ of\\
$\Aut(K\langle
x_1,\dots,x_n\rangle)$ is inner.
\end{prop}

\smallskip

One could try and generalize the free algebra version of the Bia\l{}ynicki-Birula's theorem to other noncommutative situations. Another way of generalization lies in changing the dimension of the torus. In a complete analogy with further work of Bia\l{}ynicki-Birula \cite{BB2}, we expect the following to hold.
\begin{conj} \label{conjlin}
Any effective action of $\mathbb{T}_{n-1}$ on $F_n$ is linearizable.
\end{conj}

On the other hand, there is little reason to expect this statement to hold with further lowering of the torus dimension. In fact, even in the commutative case the conjecture that any effective toric action is linearizable, in spite of considerable effort (see \cite{KR}), proved negative (counterexamples in positive characteristic due to Asanuma, \cite{Asanuma}).

Another direction would be to replace $\mathbb{T}$ by an arbitrary reductive algebraic group, however the commutative case also does not hold even in characteristic zero (cf. \cite{Sch}).


\begin{thebibliography} {0}

\bibitem{BB1} {\sl A. Bia\l{}ynicki-Birula}, {\it Remarks on the
 action of an algebraic torus on $k^{n}$}, Bull. Acad.
Polon. Sci., Ser. Sci. Math. Astro. Phys., {\bf 14} (1966), 177-181.

\bibitem{BB2} {\sl A. Bia\l{}ynicki-Birula}, {\it Remarks on the
 action of an algebraic torus on $k^{n}$. II},  Bull. Acad. Polon. Sci. Ser. Sci. Math. Astro. Phys., {\bf 15} (1967), 123-125.

\bibitem{RefJ2}

{\sl A. Bia\l{}ynicki-Birula}, {\it Some theorems on actions of algebraic groups}, Annals of Mathematics (1973): 480-497.


\bibitem{RefB}
{\sl A. Gutwirth}, {\it The action of an algebraic torus on the affine plane}, Transactions of the American Mathematical Society {\bf 105.3} (1962): 407-414.

\bibitem{RefB}
{\sl M. Brion}, {\it Introduction to actions of algebraic groups}, Les cours du CIRM {\bf 1.1} (2010): 1-22.

\bibitem{KBYu} {\sl A. Kanel-Belov, J.-T. Yu, A. Elishev}, {\it On the augmentation topology on automorphism groups of affine spaces and algebras}, arXiv:1207.2045, to appear in IJAC.

\bibitem{KR} {\sl T. Kambayashi and P. Russell}, {\it On linearizing algebraic torus actions}, J. Pure and Applied Algebra, {\bf 23} (1982), 243-250.

\bibitem{KoRu2} {\sl M. Koras and P. Russell}, {\it $C^*$-actions on $C^3$: The smooth locus of the quotient is not of hyperbolic type.} Journal of Algebraic Geometry {\bf 8.4} (1999): 603-694.

\bibitem{Zung} {\sl N. T. Zung}, {\it Torus actions and integrable systems.}, arXiv math/0407455 (2004).

\bibitem{Asanuma} {\sl T. Asanuma}, {\it Non-linaearazible $k^*$-actions in
affine space}, Invent. Math. {\bf 138} (1999)  281-306.

\bibitem{Sch} {\sl G. Schwarz}, {\it Exotic algebraic group actions}, C. R. Acad. Sci. Paris., Ser. I {\bf 309} (1989), 89-94.

\bibitem{KGE}   {\sl A. Kanel-Belov, S. Grigoriev, A. Elishev, J.-T. Yu and W. Zhang}, {\it Lifting of Polynomial Symplectomorphisms and Deformation Quantization}, arXiv:1707.06450, Commun. in Algebra, Vol. 46 (2018), 3926-3938.

\bibitem{KBYu} {\sl A. Kanel-Belov, J.-T. Yu, A. Elishev}, {\it On the Augmentation Topology on Automorphism Groups of Affine Spaces and Algebras}, arXiv:1207.2045, Int. J. Alg. Comp., https://doi.org/10.1142/S0218196718400040.



\end{thebibliography}
\end{document}